\documentclass[11pt]{article}
\usepackage{amsmath,amssymb,amsthm}
\usepackage{amscd}
\newtheorem{theorem}{Theorem}[section]
\newtheorem{lemma}[theorem]{Lemma}
\newtheorem{proposition}[theorem]{Proposition}
\newtheorem{corollary}[theorem]{Corollary}

\theoremstyle{plain}

\theoremstyle{definition}

\numberwithin{equation}{section}

\renewcommand{\labelenumi}{\textup{(\theenumi)}}

\newcommand{\Homeo}{\operatorname{Homeo}}
\newcommand{\id}{\operatorname{id}}
\newcommand{\Ker}{\operatorname{Ker}}
\newcommand{\Ad}{\operatorname{Ad}}

\def\det{{{\operatorname{det}}}}

\newcommand{\N}{\mathbb{N}}

\newcommand{\R}{\mathbb{R}}

\newcommand{\Z}{\mathbb{Z}}
\newcommand{\Zp}{{\mathbb{Z}}_+}

\title{Continuous orbit equivalence of topological Markov shifts \\
and KMS states on Cuntz--Krieger algebras}
\author{Kengo Matsumoto \\
Department of Mathematics \\
Joetsu University of Education \\
Joetsu, 943-8512, Japan
}

\begin{document}
\maketitle

\date{}

\def\det{{{\operatorname{det}}}}

\begin{abstract}
We study KMS states for gauge actions with potential functions
on Cuntz--Krieger algebras whose underlying one-sided topological Markov shifts are 
continuously orbit equivalent.
As a result,
we have a certain relationship between topological entropy of   
continuously orbit equivalent one-sided topological Markov shifts.
\end{abstract}




\def\OA{{{\mathcal{O}}_A}}
\def\OaA{{{\mathcal{O}}^a_A}}
\def\OB{{{\mathcal{O}}_B}}
\def\FA{{{\mathcal{F}}_A}}
\def\FB{{{\mathcal{F}}_B}}
\def\DA{{{\mathcal{D}}_A}}
\def\DB{{{\mathcal{D}}_B}}
\def\Ext{{{\operatorname{Ext}}}}
\def\Max{{{\operatorname{Max}}}}
\def\Per{{{\operatorname{Per}}}}
\def\PerB{{{\operatorname{PerB}}}}
\def\Homeo{{{\operatorname{Homeo}}}}
\def\HA{{{\frak H}_A}}
\def\HB{{{\frak H}_B}}
\def\HSA{{H_{\sigma_A}(X_A)}}
\def\Out{{{\operatorname{Out}}}}
\def\Aut{{{\operatorname{Aut}}}}
\def\Ad{{{\operatorname{Ad}}}}
\def\Inn{{{\operatorname{Inn}}}}
\def\det{{{\operatorname{det}}}}
\def\exp{{{\operatorname{exp}}}}
\def\cobdy{{{\operatorname{cobdy}}}}
\def\Ker{{{\operatorname{Ker}}}}
\def\ind{{{\operatorname{ind}}}}
\def\id{{{\operatorname{id}}}}
\def\supp{{{\operatorname{supp}}}}
\def\co{{{\operatorname{co}}}}
\def\Sco{{{\operatorname{Sco}}}}
\def\Act{{{\operatorname{Act}_{\DA}(\mathbb{R},\OA)}}}
\def\U{{{\operatorname{U}}}}

\section{Introduction}

Let $A=[A(i,j)]_{i,j=1}^N$ 
be a square matrix with entries in $\{0,1\}$,
where $1< N \in {\Bbb N}$.
Throughout the paper, 
we assume that 
$A$ is irreducible and not a permutation.
The shift space
$X_A$ of a one-sided topological Markov shift 
$(X_A,\sigma_A)$
consists of one-sided sequences
$(x_n )_{n \in \N} \in \{1,\dots,N \}^{\N}$
satisfying
$
A(x_n,x_{n+1}) =1 
$
for all 
$ n \in {\N}$.
Take and fix $\theta \in \R$ satisfying $0 <\theta <1$.
Define a metric $d_\theta$ on $X_A$ by 
$d_\theta(x,y) = \theta^k$ 
for $x=(x_n)_{n \in \N},y=(y_n)_{n \in \N}$ with $x\ne y$
where $k$ is the largest non-negative integer such that
$x_n = y_n, n<k$.
The space $X_A$ is a zero-dimensional compact Hausdorff space
by the metric $d_\theta$.
The shift transformation $\sigma_A: X_A\longrightarrow X_A$ 
is defined by 
a continuous surjection satisfying
 $\sigma_{A}((x_n)_{n \in \Bbb N})=(x_{n+1} )_{n \in \Bbb N}$.
The two-sided topological Markov shift
 $(\bar{X}_A, \bar{\sigma}_A)$
is a topological dynamical system of a homeomorphism
$\bar{\sigma}_A((x_n)_{n\in \Z})) =(x_{n+1})_{n\in \Z}$
on the zero-dimensional compact Hausdorff space  
$\bar{X}_A$
consisting of two-sided sequences 
$(x_n)_{n \in \Z}$ 
satisfying
$
A(x_n,x_{n+1}) =1 
$
for all 
$ n \in {\Z}$.

The Cuntz--Krieger algebra $\OA$ for the matrix $A$ is defined to be 
a universal unique $C^*$-algebra generated by
partial isometries
$S_1,\dots,S_N$
satisfying the relations:
$ 
\sum_{j=1}^N S_j S_j^* = 1, \,
S_i^* S_i = \sum_{j=1}^N A(i,j) S_jS_j^*, \, i=1,\dots,N (\cite{CK}). 
$
The gauge automorphisms
$\gamma^A_t, t \in {\mathbb{R}}$
on $\OA$
are defined by 
$\gamma^A_t(S_j)= e^{\sqrt{-1}t}S_j, j=1,\dots,N.$
They 
yield an  action of $\R$
on $\OA$ which we call the gauge action.
A word $\mu =(\mu_1, \dots, \mu_k)$ 
for $\mu_i \in \{1,\dots,N\}$
is said to be admissible for
the topological Markov shift
$(X_A,\sigma_A)$
if $\mu$ appears somewhere in an element $x$ in $X_A$.
The length of $\mu$ is $k$, 
which is denoted by $|\mu|$.
 We denote by 
$B_k(X_A)$ the set of all admissible words of length $k$.
We set 
$B_*(X_A) = \cup_{k=0}^\infty B_k(X_A)$ 
where $B_0(X_A)$ denotes  the empty word $\emptyset$.
Denote by $U_\mu$ the cylinder set 
$\{ (x_n )_{n \in \Bbb N} \in X_A 
\mid x_1 =\mu_1,\dots, x_k = \mu_k \}$
for $\mu=(\mu_1,\dots,\mu_k) \in B_k(X_A)$.  
The set of cylinder sets form a basis of the topology 
of $X_A$.
Let us denote by
 $\DA$
the $C^*$-subalgebra of $\OA$
generated by 
the projections of the form
$S_{\mu_1}\cdots S_{\mu_n}S_{\mu_n}^* \cdots S_{\mu_1}^*
$
for
${\mu_1\cdots \mu_n} \in B_*(X_A)$.
It is well-known that 
the commutative 
$C^*$-algebra $C(X_A)$
of complex valued continuous functions on $X_A$
is regarded as the  
$C^*$-subalgebra $\DA$
by identifying
the characteristic functions
$\chi_{U_{\mu_1\cdots \mu_n}} \in C(X_A)$
of the cylinder sets
$U_{\mu_1\cdots \mu_n}$
with
the projections 
$S_{\mu_1}\cdots S_{\mu_n}S_{\mu_n}^* \cdots S_{\mu_1}^*
\in \DA.$


Let $\alpha$ be an action of $\R$ on $\OA$.
Denote by $\OaA$ the set of analytic elements of the action. 
For a positive real number $\beta \in \R$,
a state $\psi$ on $\OA$ is called a $\log{\beta}$-{\it KMS state} 
for the action $\alpha$ 
if $\psi$ satisfies the  condition
\begin{equation}
\psi(y\alpha_{i\log{\beta}}(x)) = \psi(xy),\qquad
x \in \OaA, \, \, y \in \OA \label{eq:KMS}.
\end{equation} 
The studies of KMS states on operator algebras are very crucial from the viewpoints of quantum statistical mechanics and the structure theory of $C^*$-algebras and von Neumann algebras.
There have been many important and interesting studies
as in the text book by Bratteli--Robinson (\cite{BR}).
For Cuntz--Krieger algebras,
Enomoto--Fujii--Watatani proved that 
the gauge action $\gamma^A$ on $\OA$ 
has a $\log \beta$-KMS state if and only if $\beta = r_A$
the Perron--Frobenius eigenvalue of the matrix $A$,
and the admitted KMS state is faithful and unique 
(\cite{EFW}, see \cite{OP} for Cuntz algebras).
The reciprocal of the eigenvalue $r_A$ is the radius of convergence of the zeta function
$
\zeta_A(z)
= \exp\left(\sum_{n=1}^\infty \frac{\Per_n(\bar{X}_A)}{n} z^n\right)
$
of the topological Markov shift $(\bar{X}_A, \bar{\sigma}_A)$,
and the value $\log r_A$ is topological entropy of the topological 
Markov shift $(\bar{X}_A, \bar{\sigma}_A)$ (\cite{Parry}, cf, \cite{LM}).

In \cite{MaJOT2015},
the author studied generalized gauge actions
from the viewpoints of continuous orbit equivalence and flow equivalence of 
topological Markov shifts 
(see also \cite{BC}, \cite{CEOR}, \cite{MaMZ2017}, \cite{MaPAMS2017}, \cite{MMETDS}, etc.).
We regard a function $f$ in $C(X_A,\R)$ 
as an element of the subalgebra $C(X_A, \mathbb{C})(=\DA)$ of $\OA$.
For $f \in C(X_A, \R)$,
define one-parameter unitaries
$V_t(f)\in \DA, t \in \R$ 
by setting $V_t(f) = \exp({\sqrt{-1} t f})$
and an automorphism
$\gamma_t^f$ on $\OA$ for each $t \in \R$ 
by 
\begin{equation}
\gamma_t^f(S_j) = V_t(f) S_j, \qquad j=1,\dots,N. \label{eq:gammatf}
\end{equation}
The automorphisms
$\gamma_t^f, t \in \R$ yield an action of 
$\R$ on $\OA$
 such that
$\gamma^f_t(a) =a$
for all $a \in \DA$.
Such an action on $\OA$ is called a generalized gauge action in \cite{ExelBBMS2004}.
If the function $f$ is constantly $1$,
the action $\gamma^1$ is the gauge action 
$\gamma^A$.
Let $F_\theta(X_A)$
be the set of real valued 
$\theta$-H\"{o}lder continuous functions 
on $X_A$.
In \cite{ExelJFA2003}, \cite{ExelBBMS2004} and  \cite{ExelLopes},
R. Exel and Exel--Lopes have studied KMS states for generalized gauge actions
on the $C^*$-algebras constructed from crossed products by endomorphisms
including Cuntz-Krieger algebras 
from the viewpoint of thermodynamic formalism 
of dynamical systems 
and shown that 
there exists a bijective correspondence between KMS states 
and eigenvectors of  Ruelle operators (cf. \cite{PWY}).
As a result, Exel \cite{ExelBBMS2004} showed that  
there exists a unique KMS state for a generalized gauge action on Cuntz--Krieger algebras.
For $\phi \in F_\theta(X_A)$,
the Ruelle operator $\lambda_\phi:C(X_A)\rightarrow C(X_A)$
on a topological Markov shift
$(X_A,\sigma_A)$
is defined by 
\begin{equation}
\lambda_\phi(a) = \sum_{i=1}^N  S_i^* a e^\phi S_i, \qquad
a \in C(X_A) \label{eq:ruelleop} 
\end{equation}
where $e^\phi\in \DA$
is defined by $(e^\phi)(x) = e^{\phi(x)}, x \in X_A$
(cf. \cite{PP}, \cite{Ruelle1978}).
The Ruelle operators have been playing a key role
in Ruelle's thermodynamic formalism in topological Markov shifts
(\cite{Ruelle1978}, \cite{Ruelle2002}, cf. \cite{Baladi1998}, \cite{Bo}, \cite{PP}, etc.).

In the first part of the paper, 
we will give a direct proof 
for Exel's result above in the case of Cuntz--Krieger algebras
which says that there exists a bijective correspondence between KMS states 
and eigenvectors of the Ruelle operators
(Proposition \ref{prop:KMSRuelle}).

In the second part of the paper, which is a main part of the paper,
we will find a  certain relationship between topological entropy of   
continuously orbit equivalent topological Markov shifts, 
by using Exel's result.
The author in \cite{MaPacific}
has introduced a notion of continuous orbit equivalence 
of one-sided topological Markov shifts.
 It is weaker than one-sided topological conjugacy
 and  gives rise to isomorphic Cuntz--Krieger algebras 
(\cite{MaMZ2017}, see also \cite{MMKyoto}, \cite{BC}, etc.).
Let $A$ and $B$ be irreducible square matrices with entries in $\{0,1 \}$.
If there exists a homeomorphism
$h: X_A \rightarrow X_B$ 
such that
\begin{align}
\sigma_B^{k_1(x)} (h(\sigma_A(x))) 
& = \sigma_B^{l_1(x)}(h(x))
\quad 
\text{ for} \quad 
x \in X_A,  \label{eq:orbiteq1x} \\
\sigma_A^{k_2(y)} (h^{-1}(\sigma_B(y))) 
& = \sigma_A^{l_2(y)}(h^{-1}(y))
\quad 
\text{ for } \quad 
y \in X_B \label{eq:orbiteq2y}
\end{align}
for some continuous functions 
$k_1,l_1 \in C(X_A, \Zp),\, 
 k_2,l_2 \in C(X_B, \Zp), 
$
the one-sided topological Markov shifts
$(X_A, \sigma_A)$ and $(X_B,\sigma_B)$ 
are said to be {\it continuously orbit equivalent},
where $\Zp =\{0,1,2,\dots \}$.
The functions 
$c_1(x) = l_1(x) -k_1(x), x \in X_A$
and
$c_2(y) = l_2(y) -k_2(y), y \in X_B$
are called the cocycle function for $h$
and
the cocycle function for $h^{-1}$,
respectively.
In \cite{MMETDS},
Matui and the author have shown that
the zeta functions of continuous orbit equivalent topological Markov shifts
$(X_A,\sigma_A)$ and $(X_B,\sigma_B)$ 
have a certain relationship  by using the above cocycle functions
(see also \cite[Theorem 4.6]{MaPAMS2016}).
It suggests that the topological entropy of the topological Markov shifts
$(X_A,\sigma_A)$ and $(X_B,\sigma_B)$
have some relation,
because the topological entropy  
are  maximum poles of their zeta functions.
In this  paper, we will show the following theorem that tells us 
a relationship between topological entropy of continuously orbit 
equivalent topological Markov shifts.
\begin{theorem}[{Theorem \ref{thm:entropy}}]
Let $A$ and $B$ be irreducible, non-permutation matrices with entries in $\{0,1\}$.
Suppose that
one-sided topological Markov shifts
$(X_A, \sigma_A)$ and $(X_B,\sigma_B)$
are continuously orbit equivalent.
Let $\varphi_A$ and  $\varphi_B$
be the unique KMS states for the gauge actions on $\OA$ and on  $\OB$,
respectively.
Let
$r_A$ and  $r_B$ be the Perron--Frobenius eigenvalues of the matrix $A$ 
and of the matrix $B$, respectively.
Denote by  $h_{top}(\sigma_A)$ and  $h_{top}(\sigma_B)$
the topological entropy of  $(X_A,\sigma_A)$
and $(X_B,\sigma_B)$, respectively.
Then we have
\begin{align}
h_{top}(\sigma_A) & = -\lim_{n\to\infty} \frac{1}{n}\log \varphi_A(r_B^{-c_1^n}),
\qquad  \label{eq:th1.1A}\\
h_{top}(\sigma_B) & = -\lim_{n\to\infty} \frac{1}{n}\log \varphi_B(r_A^{-c_2^n}),\label{eq:th1.1B}
\end{align}
where 
$r_B^{-c_1^n} \in C(X_A)$ is defined by
$r_B^{-c_1^n}(x) = r_B^{-{\sum_{k=0}^{n-1}c_1(\sigma_A^k(x))}}, \, x \in X_A$,
and 
$r_A^{-c_2^n} \in C(X_B)$ is similarly defined.
\end{theorem}

In \cite{MaJOT2015}, a notion of strongly continuous orbit equivalence between 
one-sided topological Markov shifts
$(X_A, \sigma_A)$ and $(X_B,\sigma_B)$
was introduced.
It is defined as the cases where the cocycle function $c_1$ is cohomologous to 
$1$ in $C(X_A,\Z)$. 
Although we have already known that strongly continuous orbit equivalence
implies topological conjugacy of their two-sided topological Markov shifts
(\cite[Theorem 5.5]{MaJOT2015}) and hence $h_{top}(\sigma_A)=h_{top}(\sigma_B)$, 
as a direct corollary of the above theorem, we have 
\begin{corollary}[{Corollary \ref{cor:SCOE}}]
If one-sided topological Markov shifts
$(X_A, \sigma_A)$ and $(X_B,\sigma_B)$ are strongly continuously orbit equivalent, then we have
$h_{top}(\sigma_A)=h_{top}(\sigma_B)$.
\end{corollary} 
In the final section, we will 
concretely calculate the formulas \eqref{eq:th1.1A}, \eqref{eq:th1.1B}
for two matrices
\begin{equation*}
A=
\begin{bmatrix}
1 & 1 \\
1 & 1
\end{bmatrix},
\qquad
B=
\begin{bmatrix}
1 & 1 \\
1 & 0 
\end{bmatrix}
\end{equation*}
for which $(X_A, \sigma_A)$ and $(X_B,\sigma_B)$
are continuously orbit equivalent.

\section{Ruelle operators and KMS states}

In this section, we fix a topological Markov shift
$(X_A,\sigma_A)$.  
For a function $f \in C(X_A,\Z)$ and $n \in \N$,
we define a function $f^n$ on $X_A$ by
\begin{equation*}
f^n(x) = \sum_{i=0}^{n-1}f(\sigma_A^i(x)), \qquad x \in X_A.
\end{equation*}
We fix generating partial isometries $S_1,\dots,S_N$ of the Cuntz--Krieger
algebra $\OA$ satisfying the relations:
$ 
\sum_{j=1}^N S_j S_j^* = 1, \,
S_i^* S_i = \sum_{j=1}^N A(i,j) S_jS_j^*, \, i=1,\dots,N (\cite{CK}). 
$
For a word $\mu = (\mu_1,\dots,\mu_n)\in B_n(X_A)$,
denote by 
$S_\mu $ the partial isometry
$S_{\mu_1}\cdots S_{\mu_n}$.
The $C^*$-subalgebra $\DA$ generated by 
the projections
$S_\mu S_\mu^*, \mu \in B_*(X_A)$ 
is identified with 
the commutative 
$C^*$-algebra $C(X_A)$
of the complex valued continuous functions on $X_A$
through the correspondence
$S_\mu S_\mu^* \leftrightarrow \chi_{U_\mu}$,
where
$\chi_{U_{\mu}} \in C(X_A)$
is the characteristic function of of the cylinder set
$U_{\mu}$.
Recall that $0<\theta<1$
is a real number defining the metric $d_\theta(x,y), \, x, y \in X_A$  
on $X_A$.
A continuous function
$\phi:X_A \rightarrow \R$ 
is said to be $\theta$-H\"{o}lder continuous
if there exists a constant $C$ 
such that
$|\phi(x) - \phi(y) | \le C d_\theta(x,y)$.
We denote by $F_\theta(X_A)$ 
the set of 
$\theta$-H\"{o}lder continuous
functions on $X_A$.
It is easy to see that an integer valued continuous function 
on $X_A$ is $\theta$-H\"{o}lder continuous, 
and hence so is a scalar multiple of an integer valued continuous function 
on $X_A$.
Under our identification
between $\DA$ and $C(X_A)$ above,
the Ruelle operator 
$\lambda_\phi:\DA\rightarrow \DA$
for $\phi \in F_\theta(X_A)$
is defined by the formula \eqref{eq:ruelleop}. 
It may be  written
\begin{equation*}
\lambda_\phi(f) (x) 
= \sum_{\sigma_A(y) =x}e^{\phi(y)}f(y), \qquad x \in X_A.
\end{equation*}
It is also called the transfer operator
(cf. \cite{Baladi1998}).
One easily sees that 
$\lambda_\phi(f) \in F_\theta(X_A)$ for $f \in F_\theta(X_A)$.
For $\phi\equiv 0$, we write
$\lambda_{\phi}$ as $\lambda_A$.
We note that
the identities
\begin{equation}
(\lambda_\phi)^n(a) = \lambda_A^n( a e^{\phi^n}),
\qquad
a \in \DA \label{eq:nruelleop} 
\end{equation}
hold,
where
$\phi^n(x) = \sum_{i=0}^{n-1}\phi(\sigma_A^i(x)), x \in X_A$.
The following  well-known result is known 
as Ruelle--Perron--Frobenius Theorem.  
\begin{lemma}[{\cite[Theorem 2.2]{PP}}] \label{lem:RPF}
For $\phi \in F_\theta(X_A)$, we have the following assertions.
\begin{enumerate}
\renewcommand{\theenumi}{\roman{enumi}}
\renewcommand{\labelenumi}{\textup{(\theenumi)}}
\item
There exists 
a unique positive eigenvalue $r_\phi$ of $\lambda_\phi$, 
a strictly positive function $g_\phi \in \DA$
and
a faithful state $\varphi_\phi$ on $\DA$
satisfying  
\begin{equation}
\lambda_\phi(g_\phi) = r_\phi g_\phi, \qquad
\varphi_\phi \circ\lambda_\phi = r_\phi \varphi_\phi, \qquad
\varphi_\phi(g_\phi) = 1. \label{eq:RPF}
\end{equation}
\item
$\frac{(\lambda_\phi)^n(a)}{r_\phi^n}$
uniformly converges to
$\varphi_\phi(a) g_\phi$
for each $a \in \DA$.
\end{enumerate}
\end{lemma}
Let us denote by
$\FA$ the $C^*$-subalgebra of $\OA$
consisting of fixed elements under the automorphisms
$\gamma^A_t, t \in \R$
of the gauge action.
The algebra is an AF-algebra  
whose diagonal elements give rise to the algebra
$\DA$.
There is a canonical conditional expectation
$E_D: \FA \rightarrow \DA$
by taking diagonal elements.
Let us denote by
$E_A:\OA \rightarrow \DA$
the canonical conditional expectation
defined by 
$E_A(T) = E_D(\frac{1}{2\pi}\int_{0}^{2\pi} \gamma^A_t(T)dt)$ 
for $T \in \OA$
where $dt$ is the Lebesgue measure on $[0,2\pi]$.

We regard a function $f$ in $C(X_A,\R)$ 
as an element of the subalgebra $C(X_A, \mathbb{C})(=\DA)$ of $\OA$.
For $f \in C(X_A, \R)$,
recall that 
a one-parameter automorphism group
$\gamma_t^f, t \in \R$ on $\OA$ 
is defined by 
\begin{equation*}
\gamma_t^f(S_j) = V_t(f) S_j, \quad j=1,\dots,N,
\quad
\text{ where }
V_t(f) = \exp({i t f}).
\end{equation*}
It is easy to see that
the automorphisms
$\gamma_t^f, t \in \R$ yield an action of 
$\R$ on $\OA$
 such that
$\gamma^f_t(a) =a$
for all $a \in \DA$.
If the function $f$ is constantly $1$,
the action $\gamma^1$ is the gauge action 
$\gamma^A$.
Through the identification
between
$\chi_{U_\mu} \in C(X_A)$
and
$S_\mu S_\mu^* \in \DA$,
we have
$ \sum_{i=1}^N S_i f S_i^* = f \circ \sigma_A$
so that the equalities
\begin{equation*}
S_i V_t(f) 
= \sum_{j=1}^N S_j V_t(f) S_j^* S_i 
= V_t(f\circ\sigma_A) S_i, 
\quad i=1,\dots,N 
\end{equation*}
hold.
The identity
\begin{equation*}
\gamma_t^f(S_\mu) = V_t(f^n)S_\mu,
\qquad \mu = (\mu_1,\dots,\mu_n) \in B_n(X_A), \, \,  t \in \R
\end{equation*}
is directly shown (cf. \cite[Lemma 3.1]{MaMZ2017})
and useful in our proof of Proposition \ref{prop:KMSRuelle} below.

Let us denote by
${\frak S}(\OA)$ the state space of $\OA$
which is the set of continuous positive linear functionals $\varphi$
on $\OA$ satisfying $\varphi(1) = 1$.
 We similarly define the state space 
 ${\frak S}(\DA)$ of $\DA$.
Let $\alpha$ be an action of $\R$ on $\OA$.
Denote by $\OaA$ the set of analytic elements of the action 
$\alpha$ (\cite[8.12]{Pedersen}).
Following after \cite{BR}, 
for a real number $1<\beta \in \R$,
a state 
$\psi \in {\frak S}(\OA)$ is called a $\log{\beta}$-{\it KMS state} 
for the $\alpha$-action 
if $\psi$ satisfies the  condition
\eqref{eq:KMS}.
The equality \eqref{eq:KMS}
is called the KMS condition.

The following proposition
has been seen in Exel's paper \cite{ExelBBMS2004}
in a slightly different form.
We will give a direct proof which is different from 
those of \cite{ExelBBMS2004} and \cite{ExelJFA2003}.

\noindent
\begin{proposition}[{Exel \cite{ExelBBMS2004}}]\label{prop:KMSRuelle}
\hspace{7cm}
\begin{enumerate}
\renewcommand{\theenumi}{\roman{enumi}}
\renewcommand{\labelenumi}{\textup{(\theenumi)}}
\item
For $f \in F_\theta(X_A)$
and
$1 <\beta \in \R$, 
if $\psi \in {\frak S}(\OA)$ is a $\log\beta$-KMS state for the
$\gamma^f$-action,
by putting
$\phi = (1 - f){\log\beta}$,
the restriction 
$\varphi = \psi|_{\DA}\in {\frak S}(\DA)$ 
of $\psi$ to $\DA$
satisfies
$\varphi \circ \lambda_\phi = \beta \varphi$ on $\DA$.
\item
For $\phi \in F_\theta(X_A)$
and
$1 <\beta \in \R$, 
if $\varphi \in {\frak S}(\DA)$ is a state
satisfying
$\varphi \circ \lambda_\phi = \beta \varphi$ on $\DA$,
then by putting
$f = 1 - \frac{1}{\log\beta}\phi$,
the state
$\psi =\varphi\circ E_A \in {\frak S}(\OA)$
is a $\log\beta$-KMS state for the $\gamma^f$-action.
\end{enumerate}
\end{proposition}  
\begin{proof}
(i)
Since $e^{\phi} = \beta^{1-f} \in \DA$
and
$\beta^{-f}S_j = \gamma_{i\log{\beta}}^f(S_j)$,
we have
\begin{equation*}
\lambda_\phi(a) 
= \sum_{j=1}^N S_j^* a e^{\phi} S_j
= \beta \sum_{j=1}^N S_j^* a \beta^{-f} S_j
= \beta \sum_{j=1}^N S_j^* a \gamma_{i\log{\beta}}^f(S_j), 
\qquad a \in \DA.
\end{equation*}
Let us note that 
the set of finite linear combinations 
of elements of the form
$S_\mu S_\nu^*, \, \mu, \nu \in B_*(X_A)$ 
is dense in the analytic elements of $\OA$
for the action $\gamma^f$.
By the KMS condition \eqref{eq:KMS}, we have
\begin{equation*}
(\varphi\circ \lambda_\phi)(a) 
= \beta \sum_{j=1}^N \psi(S_j^* a \gamma_{i\log{\beta}}^f(S_j))
= \beta \sum_{j=1}^N \psi(S_j S_j^* a )
= \beta \varphi(a),
\qquad a \in \DA.
\end{equation*}

(ii)
Suppose that $\varphi \in {\frak S}(\DA)$
satisfies
$\varphi\circ \lambda_\phi = \beta \varphi$.
Put
$f = 1 - \frac{1}{\log\beta}\phi$
and 
the state
$\psi =\varphi\circ E_A \in {\frak S}(\OA)$.
Since the set of finite linear combinations 
of elements of the form
$S_\mu S_\nu^*, \, \mu, \nu \in B_*(X_A)$ 
is dense in $\OaA$
and $\psi$ vanishes outside of diagonal elements of 
$\FA$,
it suffices to show the KMS condition \eqref{eq:KMS} 
only for the following two cases and two subcases for each.

Case (1)  $x=a S_\nu^*, y = S_\nu b$.

\noindent
Subcase (1-1) $ a= S_\xi S_\eta^*$ and $b = S_\eta S_\zeta S_\zeta^* S_\xi^*$
where $|\xi| = |\eta|$.

\noindent
Subcase (1-2) $ a= S_\eta S_\zeta S_\zeta^* S_\xi^*$ and $b = S_\xi S_\eta^*$
where $|\xi| = |\eta|$.

Case (2)  $x=S_\nu b, y = a S_\nu^*$.

\noindent
Subcase (2-1) $ a= S_\xi S_\eta^*$ and $b = S_\eta S_\zeta S_\zeta^* S_\xi^*$
where $|\xi| = |\eta|$.

\noindent
Subcase (2-2) $ a= S_\eta S_\zeta S_\zeta^* S_\xi^*$ and $b = S_\xi S_\eta^*$
where $|\xi| = |\eta|$.

Throughout the above four subcases,
we put
$n = |\nu |, \, 
k = |\xi| = |\eta |, \, 
m = |\zeta|.
$
We note the identity
$
\phi^\ell + f^\ell\log\beta =\ell \log\beta$ for $\ell \in \N$
holds.
To show the equality
$\psi(y \gamma_{i\log\beta}^f(x)) =\psi(x y)$
for each of the  four subcases, we use the equality
$\varphi\circ \lambda_\phi = \beta \varphi$ of the
assumption.

For the subcase (1-1), we have
\begin{align*}
\psi(y \gamma_{i\log\beta}^f(x))
& = \frac{1}{\beta^n}
    \varphi(\lambda_\phi^n(
    S_\nu S_\eta S_\zeta S_\zeta^* S_\xi^* 
    \gamma_{i\log\beta}^f(S_\xi S_\eta^* S_\nu^*)) \\
& = \frac{1}{\beta^n}
    \varphi(S_\nu^* 
    S_\nu S_\eta S_\zeta S_\zeta^* S_\xi^* 
    e^{-f^k\log\beta} S_\xi S_\eta^* e^{f^k\log\beta} S_\nu^*e^{f^n\log\beta} 
    e^{\phi^n} S_\nu) \\
& = \frac{1}{\beta^n}
    \varphi( S_\eta S_\zeta S_\zeta^* S_\xi^* 
    e^{-f^k\log\beta} S_\xi S_\eta^* e^{f^k\log\beta} S_\nu^*
    e^{f^n\log\beta +\phi^n} 
    S_\nu)  \\   
& = \varphi( 
    S_\eta S_\zeta S_\zeta^* S_\xi^* 
    e^{-f^k\log\beta} S_\xi S_\eta^* e^{f^k\log\beta} S_\nu^* S_\nu) \\     
& = \frac{1}{\beta^k}
    \varphi(\lambda_\phi^k(
    S_\eta S_\zeta S_\zeta^* S_\xi^* 
    e^{-f^k\log\beta} S_\xi S_\eta^* e^{f^k\log\beta} S_\nu^* S_\nu)) \\     
& = \frac{1}{\beta^k}
    \varphi(S_\eta^* 
    S_\eta S_\zeta S_\zeta^* S_\xi^* 
    e^{-f^k\log\beta} S_\xi S_\eta^* e^{f^k\log\beta} S_\nu^* S_\nu
    e^{\phi^k} S_\eta)  \\    
& = \frac{1}{\beta^k}
    \varphi(S_\eta^* e^{f^k\log\beta} S_\nu^* S_\nu
    S_\eta S_\eta^* e^{\phi^k}
    S_\eta S_\zeta S_\zeta^* S_\xi^* 
    e^{-f^k\log\beta} S_\xi) \\ 
& = \frac{1}{\beta^k}
    \varphi(S_\eta^* e^{f^k\log\beta + \phi^k}
    S_\nu^* S_\nu S_\eta S_\zeta S_\zeta^* S_\xi^* 
    e^{-f^k\log\beta} S_\xi)  \\
& = \varphi(S_\eta^* 
    S_\nu^* S_\nu S_\eta S_\zeta S_\zeta^* S_\xi^* 
    e^{-f^k\log\beta} S_\xi) \\  
& = \varphi(S_\xi^* S_\xi S_\eta^* 
    S_\nu^* S_\nu S_\eta S_\zeta S_\zeta^* S_\xi^* 
    e^{-f^k\log\beta -\phi^k} e^{\phi^k} S_\xi) \\
& = \varphi(\lambda_\phi^k( S_\xi S_\eta^* 
    S_\nu^* S_\nu S_\eta S_\zeta S_\zeta^* S_\xi^* 
    e^{-f^k\log\beta -\phi^k} ))\\
& = \varphi(\lambda_\phi^k( S_\xi S_\eta^* 
    S_\nu^* S_\nu S_\eta S_\zeta S_\zeta^* S_\xi^* 
    e^{-k\log\beta} )) 
 = \psi( x y).
\end{align*}

For the subcase (1-2), we have
\begin{align*}
\psi(y \gamma_{i\log\beta}^f(x))
& = \varphi(
    S_\nu S_\xi S_\eta^*
    e^{-f^{k+m}\log\beta} S_\eta S_\zeta S_\zeta^* S_\xi^*e^{f^{k+m}\log\beta} 
    S_\nu^* e^{f^n\log\beta}) \\
& = \frac{1}{\beta^n}\varphi(\lambda_\phi^n(
    S_\nu S_\xi S_\eta^*
    e^{-f^{k+m}\log\beta} S_\eta S_\zeta S_\zeta^* S_\xi^*e^{f^{k+m}\log\beta} 
    S_\nu^* e^{f^n\log\beta})) \\
& = \frac{1}{\beta^n}\varphi(
    S_\nu^* S_\nu S_\xi S_\eta^*
    e^{-f^{k+m}\log\beta} S_\eta S_\zeta S_\zeta^* S_\xi^*e^{f^{k+m}\log\beta} 
    S_\nu^* e^{f^n\log\beta +\phi^n} S_\nu) \\
& = \varphi(
    S_\xi S_\eta^*
    e^{-f^{k+m}\log\beta} S_\eta S_\zeta S_\zeta^* S_\xi^*e^{f^{k+m}\log\beta} 
    S_\nu^* S_\nu) \\
& = \frac{1}{\beta^{k+m}}\varphi(\lambda_\phi^{k+m} (
    S_\xi S_\eta^*
    e^{-f^{k+m}\log\beta} S_\eta S_\zeta S_\zeta^* S_\xi^*e^{f^{k+m}\log\beta} 
    S_\nu^* S_\nu) ) \\
& = \frac{1}{\beta^{k+m}}\varphi(
    S_\zeta^* S_\xi^*
    S_\xi S_\eta^*
    e^{-f^{k+m}\log\beta} S_\eta S_\zeta S_\zeta^* S_\xi^*e^{f^{k+m}\log\beta} 
    S_\nu^* S_\nu e^{\phi^{m+k}} S_\xi S_\zeta) \\
& = \varphi(
    S_\zeta^* S_\eta^*
    e^{-f^{k+m}\log\beta} S_\eta S_\zeta S_\zeta^* S_\xi^* 
    S_\nu^* S_\nu  S_\xi S_\zeta) \\
& = \varphi(\lambda_\phi^{k+m}(
    e^{- \phi^{k+m} -f^{k+m}\log\beta} S_\eta S_\zeta S_\zeta^* S_\xi^* 
    S_\nu^* S_\nu  S_\xi S_\zeta S_\zeta^* S_\eta^*)) \\
& = \varphi(
    S_\eta S_\zeta S_\zeta^* S_\xi^* 
    S_\nu^* S_\nu  S_\xi S_\zeta S_\zeta^* S_\eta^*) 
 = \psi( x y).
\end{align*}

For the subcase (2-1), we have
\begin{align*}
\psi(y \gamma_{i\log\beta}^f(x))
& = \varphi(S_\xi S_\eta^* S_\nu^*
    e^{-f^n\log\beta} S_\nu e^{-f^{k+m}\log\beta} S_\eta S_\zeta S_\zeta^* S_\xi^*
    e^{f^{k+m}\log\beta}) \\
& = \frac{1}{\beta^{k+m}} \varphi(\lambda_\phi^{k+m}(S_\xi S_\eta^* S_\nu^*
    e^{-f^n\log\beta} S_\nu e^{-f^{k+m}\log\beta} S_\eta S_\zeta S_\zeta^* S_\xi^*
    e^{f^{k+m}\log\beta})) \\
& = \frac{1}{\beta^{k+m}} \varphi( S_\zeta^* S_\xi^*
    S_\xi S_\eta^* S_\nu^*
    e^{-f^n\log\beta} S_\nu e^{-f^{k+m}\log\beta} S_\eta S_\zeta S_\zeta^* S_\xi^*
    e^{f^{k+m}\log\beta + \phi^{k+m}}S_\xi S_\zeta ) \\
& = \varphi( S_\zeta^* S_\xi^*
    S_\xi S_\eta^* S_\nu^*
    e^{-f^n\log\beta} S_\nu e^{-f^{k+m}\log\beta} S_\eta S_\zeta S_\zeta^* S_\xi^*
    S_\xi S_\zeta ) \\
& = \varphi( S_\zeta^* S_\xi^*
    S_\xi S_\eta^* S_\nu^*
    e^{-f^n\log\beta} S_\nu e^{-f^{k+m}\log\beta} S_\eta S_\zeta S_\zeta^* S_\eta^* 
    S_\eta S_\zeta) \\
& = \varphi( \lambda_\phi^{k+m}(S_\eta S_\zeta  S_\zeta^* S_\xi^*
    S_\xi S_\eta^* S_\nu^*
    e^{-f^n\log\beta} S_\nu e^{-f^{k+m}\log\beta} S_\eta S_\zeta  S_\zeta^* S_\eta^*
    e^{-\phi^{k+m}})) \\
& = \beta^{k+m} \varphi(S_\eta S_\zeta  S_\zeta^* S_\xi^*
    S_\xi S_\eta^* S_\nu^*
    e^{-f^n\log\beta} S_\nu e^{-f^{k+m}\log\beta -\phi^{k+m}}) \\
& = \varphi(S_\eta S_\zeta  S_\zeta^* S_\xi^*
    S_\xi S_\eta^* S_\nu^*
    e^{-f^n\log\beta} S_\nu S_\nu^* S_\nu ) \\
& = \varphi(\lambda_\phi^n( S_\nu S_\eta S_\zeta  S_\zeta^* S_\xi^*
    S_\xi S_\eta^* S_\nu^*
    e^{-f^n\log\beta} S_\nu  S_\nu^* e^{-\phi^n} )) \\
& = \beta^n \varphi( S_\nu S_\eta S_\zeta  S_\zeta^* S_\xi^*
    S_\xi S_\eta^* S_\nu^*
    e^{-f^n\log\beta -\phi^n} S_\nu  S_\nu^* ) \\
& = \varphi( S_\nu S_\eta S_\zeta  S_\zeta^* S_\xi^*
    S_\xi S_\eta^* S_\nu^*) 
 = \psi( x y).
\end{align*}

For the subcase (2-2), we have
\begin{align*}
\psi(y \gamma_{i\log\beta}^f(x))
& = \varphi( S_\eta S_\zeta S_\zeta^* S_\xi^* S_\nu^*
    e^{-f^{k+m}\log\beta} S_\nu S_\xi S_\eta^* e^{f^k\log\beta}) \\
& = \frac{1}{\beta^k} \varphi(\lambda_\phi^k( S_\eta S_\zeta S_\zeta^* S_\xi^* S_\nu^*
    e^{-f^{k+m}\log\beta} S_\nu S_\xi S_\eta^* e^{f^k\log\beta})) \\
& = \frac{1}{\beta^k} \varphi(S_\eta^* S_\eta S_\zeta S_\zeta^* S_\xi^* S_\nu^*
    e^{-f^{k+m}\log\beta} S_\nu S_\xi S_\eta^* e^{f^k\log\beta} e^{\phi^k}S_\eta) \\
& = \varphi(S_\eta^* S_\eta S_\zeta S_\zeta^* S_\xi^* S_\nu^*
    e^{-f^{k+m}\log\beta} S_\nu S_\xi ) \\
& = \varphi( \lambda_\phi^{k+m}(S_\nu S_\xi S_\eta^* S_\eta S_\zeta S_\zeta^* S_\xi^* S_\nu^*
    e^{-f^{k+m}\log\beta -\phi^{k+m}} ) ) \\
& = \varphi( S_\nu S_\xi S_\eta^* S_\eta S_\zeta S_\zeta^* S_\xi^* S_\nu^*) 
 = \psi( x y).
\end{align*}
Therefore $\psi$ is a $\log\beta$-KMS state for the $\gamma^f$-action on $\OA$.
\end{proof} 
R. Exel proved the following fact that will be useful in our further discussions. 
\begin{lemma}[{Exel \cite{ExelBBMS2004}}]\label{lem:Exellemma}
For $f \in F_\theta(X_A),$
there exists a unique positive real number
$\beta_f \in \R$ such that 
there exists a $\log\beta_f$-KMS state for the $\gamma^f$-action 
on $\OA$.
The admitted KMS state is faithful and unique.
\end{lemma}
\section{Continuous orbit equivalence and KMS conditions}
Throughout the section,
two irreducible square matrices $A,B$ with entries in $\{0,1\}$ 
are fixed. 
The Ruelle operator $\lambda_\phi$ on $\DA$ 
is denote by $\lambda^A_\phi$,
and the action $\gamma^f$ on $\OA$ 
is denoted by $\gamma^{A,f}$.
We use similar notation $\lambda_\phi^B$ and $\gamma^{B,f}$
for the matrix $B$.

Let
$h : X_A \rightarrow X_B$
be a homeomorphism
which gives rise to a continuous orbit equivalence
between
$(X_A, \sigma_A)$
and
$(X_B,\sigma_B)$.
Take $k_1, l_1 \in C(X_A,\Zp)$
and $k_2, l_2 \in C(X_B,\Zp)$
satisfying \eqref{eq:orbiteq1x}
and \eqref{eq:orbiteq2y}, respectively.
Recall that 
$c_1\in C(X_A,\Zp)$
and 
$c_2\in C(X_B,\Zp)$
are defined by 
$c_1(x) = l_1(x) - k_1(x), x \in X_A$
and
$c_2(y) = l_2(y) - k_2(y), y \in X_B$,
respectively.
\begin{proposition}[{\cite[Corollary 3.4]{MaMZ2017}}]\label{prop:cocycleconjugacy}
Let
$h : X_A \rightarrow X_B$
be a homeomorphism
which gives rise to a continuous orbit equivalence
between
$(X_A, \sigma_A)$
and
$(X_B,\sigma_B)$.
Then there exists an isomorphism
$\Phi:\OA \rightarrow \OB$ such that 
$\Phi(\DA) = \DB$ and
\begin{equation}
\Phi \circ  \gamma^{A}_t 
 = \gamma^{B,c_2}_t \circ \Phi
\quad
\text{ and }
\quad
\Phi^{-1} \circ  \gamma^{B}_t 
 = \gamma^{A,c_1}_t \circ \Phi^{-1}
\qquad
\text{ for } \,\,
t \in \R. \label{eq:3.1}
\end{equation}
\end{proposition}
Recall that $r_A, r_B$ denote the Perron-Frobenius eigenvalues of $A, B$, respectively.
\begin{lemma}\label{lem:COE3.2}
Suppose that one-sided topological Markov shifts
$(X_A, \sigma_A)$ and $(X_B,\sigma_B)$
are continuously orbit equivalent via a homeomorphism 
$h:X_A\longrightarrow X_B$.
\begin{enumerate}
\renewcommand{\theenumi}{\roman{enumi}}
\renewcommand{\labelenumi}{\textup{(\theenumi)}}
\item
There exists a $\log r_A$-KMS state for the $\gamma^{B,c_2}$-action on $\OB$,
and similarly 
there exists a $\log r_B$-KMS state for the $\gamma^{A,c_1}$-action on $\OA$.
\item
Put
$\phi_{c_2} = (1 -c_2) \log r_A \in F_\theta(X_B)$
and $\phi_{c_1} = (1 -c_1) \log r_B \in F_\theta(X_A)$. 
Then we have
 $r_{\phi_{c_2}} = r_A$ and  $r_{\phi_{c_1}} = r_B$,
where $r_{\phi_{c_2}}$
and $r_{\phi_{c_1}}$ are  positive eigenvalues of the operators
$\lambda^B_{\phi_{c_2}}$  and $\lambda^A_{\phi_{c_1}}$ satisfying \eqref{eq:RPF} on 
$\DB$ and on $\DA$, respectively.
\end{enumerate}
\end{lemma}
\begin{proof}
(i) By \cite{EFW}, there exists a unique 
$\log r_A$-KMS state written $\varphi_A$ for the gauge action on $\OA$. 
Let $\varPhi:\OA\longrightarrow \OB$ be the isomorphism satisfying \eqref{eq:3.1}.
We then have
\begin{equation*}
(\varphi_A\circ \Phi^{-1}) (\Phi(y)\Phi(\gamma^A_{i\log r_A}(x)))
=
(\varphi_A\circ \Phi^{-1}) (\Phi(x) \Phi(y)),
\qquad
x \in \OaA, \, y \in \OA. 
\end{equation*}
Since
$\Phi(\gamma^A_{i\log r_A}(x)) = \gamma^{B,c_2}_{i\log r_A}(\varPhi(x)), x \in \OaA$,
the state
$\varphi_A\circ \varPhi^{-1}$ is 
a $\log r_A$-KMS state for the $\gamma^{B,c_2}$-action on $\OB$.
We similarly know that  
there exists a $\log r_B$-KMS state for the $\gamma^{A,c_1}$-action on $\OA$.

(ii)
Take a strictly positive function  $g_{\phi_{c_2}}\in \DB$ and 
a faithful state $\varphi_{\phi_{c_2}}$ on $\DB$
satisfying \eqref{eq:RPF} for the Ruelle operator
$\lambda^B_{\phi_{c_2}}$.
By Lemma \ref{lem:Exellemma},  
the state 
$\varphi_A\circ \varPhi^{-1}$ in (i) is 
a unique $\log r_A$-KMS state for the $\gamma^{B,c_2}$-action on $\OB$.
Put
$\varphi' = \varphi_A\circ \varPhi^{-1}|_{\DB}$.
By Proposition \ref{prop:KMSRuelle},
KMS-states and normalized positive eigenvectors of the Ruelle operator bijectively correspond 
so that we have
\begin{equation*}
\varphi' \circ \lambda^B_{\phi_{c_2}} = r_A \varphi'.
\end{equation*}
As in Lemma \ref{lem:Exellemma}, a KMS state is unique, so that 
$\varphi_{\phi_{c_2}} = \varphi'$. 
Hence we have 
\begin{equation*}
\varphi_{\phi_{c_2}} \circ \lambda^B_{\phi_{c_2}} = r_A \varphi_{\phi_{c_2}}.
\end{equation*}
By \eqref{eq:RPF}, we have
\begin{equation*}
r_{\phi_{c_2}} = r_{\phi_{c_2}}\varphi_{\phi_{c_2}}(g_{\phi_{c_2}})
               = \varphi_{\phi_{c_2}}(\lambda^B_{\phi_{c_2}}(g_{\phi_{c_2}}))
               = r_A \varphi_{\phi_{c_2}}(g_{\phi_{c_2}})
               = r_A.
\end{equation*}
\end{proof}
We thus have the following proposition that states an asymptotic relation between
two Perron-Frobenius eigenvalues $r_A$ and $r_B$.
\begin{proposition}\label{prop:limitthm}
Suppose that
one-sided topological Markov shifts
$(X_A, \sigma_A)$ and $(X_B,\sigma_B)$
are continuously orbit equivalent.
Let $\varphi_A$ (resp.  $\varphi_B$)
be the unique KMS state for the gauge action of $\OA$ (resp. $\OB$).
Let $r_A$ (resp. $r_B$) be the Perron-Frobenius eigenvalue of the matrix $A$ 
(resp. $B$).
Then there exist positive constants 
$C_A, C_B$ such that
\begin{equation*}
\lim_{n\to\infty}\varphi_A({r_A^n}{r_B^{-c_1^n}}) = C_A,
\qquad
\lim_{n\to\infty}\varphi_B({r_B^n}{r_A^{-c_2^n}}) = C_B
\end{equation*}
where
\begin{align*}
r_B^{-c_1^n} \in \DA &
\text{ is defined by }
r_B^{-c_1^n}(x) = r_B^{-\sum_{k=0}^{n-1}c_1(\sigma_A^k(x))}
\text{ for } x \in X_A, \\
r_A^{-c_2^n} \in \DB &
\text{ is defined by }
r_A^{-c_2^n}(y) = r_A^{-\sum_{k=0}^{n-1}c_2(\sigma_B^k(y))}
\text{ for } y \in X_B.   
\end{align*} 
\end{proposition}
\begin{proof}
Put
$\phi_{c_2} = (1 -c_2) \log r_A$.
For the Ruelle operator
$\lambda^B_{\phi_{c_2}}$, 
take a strictly positive function  $g_{\phi_{c_2}}\in \DB$ and 
a faithful state $\varphi_{\phi_{c_2}}$ on $\DB$
satisfying \eqref{eq:RPF}.
By Lemma \ref{lem:RPF} (ii),
we have for $b \in \DB$
\begin{equation*}
\lim_{n\to\infty}
\| \frac{1}{r_{\phi_{c_2}}^n} (\lambda^B_{\phi_{c_2}})^n(b) 
-\varphi_{\phi_{c_2}}(b) g_{\phi_{c_2}} \| =0
\end{equation*}
so that
\begin{equation}
\lim_{n\to\infty}
\frac{1}{r_{\phi_{c_2}}^n} 
 \varphi_B(
(\lambda^B_{\phi_{c_2}})^n(b)) 
=
\varphi_{\phi_{c_2}}(b) \varphi_B(g_{\phi_{c_2}}). \label{eq:limvarB}
\end{equation}
By \eqref{eq:nruelleop},
we see
\begin{align*}
(\lambda^B_{\phi_{c_2}})^n(b) 
& = \lambda_B^n(e^{
{\phi_{c_2}} +{\phi_{c_2}}\circ \sigma_B + \cdots 
             + {\phi_{c_2}}\circ \sigma_B^{n-1}}   b) \\ 
& = \lambda_B^n(
r_A^{(1- {c_2}) + (1- {c_2}\circ \sigma_B) 
+ \cdots + (1- {c_2}\circ \sigma_B^{n-1})}   b) \\ 
& = \lambda_B^n(
{r_A^n}{r_A^{-c_2^n}} b).
\end{align*}
By \eqref{eq:limvarB}
with $r_{\phi_{c_2}} = r_A$ from Lemma \ref{lem:COE3.2},
we have
\begin{equation*}
\lim_{n\to\infty}
\frac{1}{r_A^n} 
 \varphi_B(
\lambda_B^n({r_A^n}{r_A^{-c_2^n}} b) )
=
\varphi_{\phi_{c_2}}(b) \varphi_B(g_{\phi_{c_2}}).
\end{equation*}
Since $\varphi_B\circ \lambda_B^n = r_B^n \varphi_B$,
we thus have
\begin{equation*}
\lim_{n\to\infty}
 \varphi_B(
{r_B^n}{r_A^{-c_2^n}} b) 
=
 \varphi_B(g_{\phi_{c_2}})\varphi_{\phi_{c_2}}(b), 
\qquad b \in \DB. 
\end{equation*}
By letting $b=1$,
we have
\begin{equation*}
\lim_{n\to\infty}
 \varphi_B(
{r_B^n}{r_A^{-c_2^n}}) 
=
 \varphi_B(g_{\phi_{c_2}}). 
\end{equation*}
Put $C_B=\varphi_B(g_{\phi_{c_2}})$.  
Since 
$\varphi_B$ is faithful,
the number $C_B$ is positive.
We thus have a desired equality.
By putting 
$C_A=\varphi_A(g_{\phi_{c_1}})$
for a stictly positive function
$g_{\phi_{c_1}} \in \DA$
with $\phi_{c_1}=(1-c_1)\log r_A$,
we similarly have the other equality
\begin{equation*}
\lim_{n\to\infty}
 \varphi_A(
{r_A^n}{r_B^{-c_1^n}}) 
= C_A.
\end{equation*}
\end{proof}
We now obtain the main result of this paper.
\begin{theorem}\label{thm:entropy}
Let $A$ and $B$ be irreducible, non-permutation matrices with entries in $\{0,1\}$.
Suppose that
one-sided topological Markov shifts
$(X_A, \sigma_A)$ and $(X_B,\sigma_B)$
are continuously orbit equivalent.
Let $\varphi_A$ and  $\varphi_B$
be the unique KMS states for the gauge actions of $\OA$ and of  $\OB$,
respectively.
Let
$r_A$ and  $r_B$ be the Perron--Frobenius eigenvalues of the matrix $A$ 
and of the matrix $B$, respectively.
Denote by  $h_{top}(\sigma_A)$ and  $h_{top}(\sigma_B)$
the topological entropy of  $(X_A,\sigma_A)$
and $(X_B,\sigma_B)$, respectivly.
Then we have
\begin{align*}
h_{top}(\sigma_A) & = -\lim_{n\to\infty} \frac{1}{n}\log \varphi_A(r_B^{-c_1^n}),
\qquad \\
h_{top}(\sigma_B) & = -\lim_{n\to\infty} \frac{1}{n}\log \varphi_B(r_A^{-c_2^n}).
\end{align*}
\end{theorem}
\begin{proof}
Since $h_{top}(\sigma_A) = \log r_A$ and $h_{top}(\sigma_B) = \log r_B$
(\cite{LM}, \cite{Parry}),
we get the desired equalities by Proposition \ref{prop:limitthm}.
\end{proof}
In \cite{MaJOT2015}, a notion of strongly continuously orbit equivalence between 
one-sided topological Markov shifts
$(X_A, \sigma_A)$ and $(X_B,\sigma_B)$
was introduced.
Suppose that $(X_A, \sigma_A)$ and $(X_B,\sigma_B)$ 
are continuously orbit equivalent via a homeomorphism
$h:X_A\longrightarrow X_B.$
Let $c_1: X_A\longrightarrow \Z$
be the cocycle function for $h:X_A\longrightarrow X_B.$
If there exists a continuous function 
$b_1: X_A\longrightarrow \Z$
such that 
\begin{equation}
c_1(x) = 1 + b_1(x) - b_1(\sigma_A(x)),\quad x \in X_A,
\label{eq:cocycle}
\end{equation}
then 
$(X_A, \sigma_A)$ and $(X_B,\sigma_B)$ are said to be 
strongly continuously orbit equivalent.
Although we have already known that strongly continuous orbit equivalence
of one-sided topological Markov shifts
implies topological conjugacy of their two-sided topological Markov shifts
(\cite[Theorem 5.5]{MaJOT2015}) so that $h_{top}(\sigma_A)=h_{top}(\sigma_B)$, 
we may prove it as an immediate corollary of the above theorem in the following way.
\begin{corollary}\label{cor:SCOE}
Let $A$ and $B$ be irreducible, non-permutation matrices with entries in $\{0,1\}$.
If one-sided topological Markov shifts
$(X_A, \sigma_A)$ and $(X_B,\sigma_B)$
are strongly continuously orbit equivalent, then we have
$h_{top}(\sigma_A)=h_{top}(\sigma_B)$.
\end{corollary} 
\begin{proof}
The equality \eqref{eq:cocycle}
implies
\begin{equation*}
c_1^n(x) = n + b_1(x) - b_1(\sigma_A^n(x)),\quad x \in X_A.
\end{equation*}
As the function
$b_1(x) - b_1(\sigma_A^n(x)), x \in X_A$
is bounded,
 the equality 
$h_{top}(\sigma_A) =h_{top}(\sigma_B)
$ 
follows from Theorem \ref{thm:entropy}.
\end{proof}


\section{Example}

Let $A$ and $B$ be the matrices:
\begin{equation*}
A=
\begin{bmatrix}
1 & 1 \\
1 & 1
\end{bmatrix},
\qquad
B=
\begin{bmatrix}
1 & 1 \\
1 & 0 
\end{bmatrix},
\end{equation*}
so that
$r_A = 2, \, r_B = \frac{1 + \sqrt{5}}{2}.$
They are both irreducible and not permutations.
The subshift $X_A$ is the on-sided full shift over $\{1,2\}$,
whereas $X_B$ is the one-sided subshift over $\{1,2\}$ forbidden
the word $(2,2)$.
Let 
$h:X_A \rightarrow X_B$
be a homeomorphism defined by 
substituting the word $(2,1)$ for $2$ in $X_A$
such as
$$
h(1,2,1,2,2,1,1,1,1,2,2,2,1,\dots )
=(1,2,1,1,2,1,2,1,1,1,1,1,2,1,2,1,2,1,1,\dots ).
$$
Then the one-sided topological Markov shifts 
$(X_A,\sigma_A)$ and $(X_B,\sigma_B)$
are continuously orbit equivalent as in \cite{MaPacific}
via the homeomorphism $h:X_A \rightarrow X_B$.
For $x = (x_n)_{n \in \N}, y = (y_n)_{n \in \N}$,
we put
\begin{equation*}
k_1(x) =
\begin{cases} 
0 & \text{ if } x_1 =1, \\
0 & \text{ if } x_1 =2,
\end{cases}
\qquad
l_1(x) =
\begin{cases} 
1 & \text{ if } x_1 =1, \\
2 & \text{ if } x_1 =2,
\end{cases}
\end{equation*}
and
\begin{equation*}
k_2(y) =
\begin{cases} 
0 & \text{ if } y_1 =1, \\
1 & \text{ if } y_1 =2,
\end{cases}
\qquad
l_2(y) =
\begin{cases} 
1 & \text{ if } y_1 =1, \\
1 & \text{ if } y_1 =2.
\end{cases}
\end{equation*}
The functions $k_1, l_1 \in C(X_A,\Zp), k_2,l_2\in C(X_B,\Zp)$
satisfy the equalities 
\eqref{eq:orbiteq1x}, \eqref{eq:orbiteq2y}. 
Hence we have the cocycle functions
\begin{equation*}
c_1(x) 
=
{\begin{cases} 
1 & \text{ if } x_1 =1, \\
2 & \text{ if } x_1 =2,
\end{cases}} \qquad 
c_2(y) 
 =
{\begin{cases} 
1 & \text{ if } y_1 =1, \\
0 & \text{ if } y_1 =2.
\end{cases}} 
\end{equation*}
We will ensure the equalities 
\eqref{eq:th1.1A} and \eqref{eq:th1.1B} in the following concrete way.

{\bf 1. Equality \eqref{eq:th1.1A}}.

As $c_1(\sigma_A^k(x)) = x_{k+1}$ 
for $x = (x_n)_{n \in \N}\in X_A$,
we have 
$$
c_1^n(x) = \sum_{i=0}^{n-1} c_1(\sigma_A^i(x)) = \sum_{i=1}^{n}x_i.
$$
It then follows that
\begin{equation*}
r_B^{-c_1^n(x)} 
= r_B^{-\sum_{i=1}^{n}x_i} 
= \sum_{\mu=(\mu_1,\dots,\mu_n)\in B_n(X_A)} r_B^{-(\mu_1+\dots+\mu_n)} \chi_{U_\mu}(x).
\end{equation*}
Sine the restriction of the KMS state $\varphi_A$ to the subalgebra
$\DA$ is the Bernoulli measure, we see that   
$\varphi_A(\chi_{U_\mu}) = \frac{1}{2^n}$ 
for 
$\mu \in B_n(X_A)$.
Therefore we have
\begin{equation*}
\varphi_A(r_B^{-c_1^n}) 
= \frac{1}{2^n}
\sum_{\mu=(\mu_1,\dots,\mu_n)\in B_n(X_A)} r_B^{-(\mu_1+\dots+\mu_n)}.
\end{equation*}
Now we have
\begin{align*}
\sum_{\mu=(\mu_1,\dots,\mu_n)\in B_n(X_A)} r_B^{-(\mu_1+\dots+\mu_n)}
& = \frac{1}{r_B^n} + \frac{n}{r_B^{n+1}} + \cdots 
 + \frac{\binom{n}{k}}{r_B^{n+k}} 
+ \cdots + \frac{1}{r_B^{2n}} \\
& = \frac{1}{r_B^n} ( 1 + \frac{1}{r_B})^n 
 = 1 
\end{align*}
so that 
we 
have
$\varphi_A(r_B^{-c_1^n}) = \frac{1}{2^n}$
and hence
\begin{equation*}
- \lim_{n\to\infty}\frac{1}{n}\log \varphi_A(r_B^{-c_1^n})
= 
- \lim_{n\to\infty}\frac{1}{n}\log\frac{1}{2^n}
= \log 2
=h_{top}(\sigma_A). 
\end{equation*}


{\bf 2. Equality \eqref{eq:th1.1B}}.

As $c_2(\sigma_B^k(y)) = 2-y_{k+1}$ 
for $y = (y_n)_{n \in \N}\in X_B$,
we have 
$$
c_2^n(y) = \sum_{i=0}^{n-1} c_2(\sigma_B^i(y)) 
= 2n -(y_1+\cdots+y_n).
$$
It then follows that by $r_A =2$
\begin{equation}
r_A^{-c_2^n(y)} 
= r_A^{\sum_{i=1}^{n}y_i -2n} 
= \frac{1}{2^{2n}}
   \sum_{\nu=(\nu_1,\dots,\nu_n)\in B_n(X_B)} 2^{(\nu_1+\dots+\nu_n)} \chi_{U_\nu}(y).
   \label{eq:rAc2}
\end{equation}
Let
$S_1, S_2$ 
be the canonical generating partial isometries of the Cuntz--Krieger algebra
$\OB$ which satisfies
\begin{equation*}
1 = S_1 S_1^* + S_2 S_2^* = S_1^* S_1, \qquad S_2^* S_2 = S_1 S_1^*. 
\end{equation*}
By the identification of the function 
$\chi_{U_\nu}$ with the projection 
$S_\nu S_\nu^*$,
the equality \eqref{eq:rAc2} tells us  that
\begin{equation*}
\varphi_B(r_A^{-c_2^n}) 
= \frac{1}{2^{2n}}
   \sum_{\nu=(\nu_1,\dots,\nu_n)\in B_n(X_B)} 2^{(\nu_1+\dots+\nu_n)} 
\varphi_B(S_\nu S_\nu^*).
\end{equation*}
As the value
$\varphi_B(S_\nu S_\nu^*)$ for $\nu \in B_n(X_B)$
is not equal to $\frac{1}{2^n}$,
we can not use an analogous method to
the computation of $\varphi_A(S_\mu S_\mu^*)$ above.
We then use an operator  
$(S_1^* + \sqrt{2}S_2^*)^n
                  (S_1 + \sqrt{2}S_2)^n$
in the following way.
By the identities
$\varphi_B(S_\nu S_\nu^*) 
= \frac{1}{r_B^n}\varphi_B(S_\nu^* S_\nu)$
for $\nu \in B_n(X_B)$
and 
\begin{equation*}
\sum_{\nu=(\nu_1,\dots,\nu_n)\in B_n(X_B)} 
   2^{(\nu_1+\dots+\nu_n)} 
   \varphi_B(S_\nu^* S_\nu)
=  \varphi_B\left( (\sqrt{2}S_1^* + 2S_2^*)^n
                  (\sqrt{2}S_1 + 2S_2)^n \right), 
\end{equation*}
we  have
\begin{equation*}
   \sum_{\nu=(\nu_1,\dots,\nu_n)\in B_n(X_B)} 
   2^{(\nu_1+\dots+\nu_n)} 
   \varphi_B(S_\nu S_\nu^*)
 =\frac{2^n}{r_B^n} 
      \varphi_B\left( (S_1^* + \sqrt{2}S_2^*)^n
                  (S_1 + \sqrt{2}S_2)^n \right). 
\end{equation*}
Put
$H_n =  (S_1^* + \sqrt{2}S_2^*)^n
                  (S_1 + \sqrt{2}S_2)^n.$
It is easy to see that 
$H_n$ is written as 
\begin{equation*}
H_n = \frac{1}{3}\{  2^{n+1} + (-1)^n\} + \frac{1}{3}\{  2^{n+1} + 2(-1)^{n-1}\}S_1S_1^*.   
\end{equation*}
As $\varphi_B$ is a KMS state on $\OB$, 
the equality 
$
1 = \varphi_B(S_1^*S_1) = r_B \varphi_B(S_1 S_1^*)  
$
holds so that 
$ \varphi_B(S_1 S_1^*) = \frac{1}{r_B}$.
It follows that
\begin{align*}
\varphi_B(r_A^{-c_2^n}) 
& = \frac{1}{2^{2n}}
    \frac{2^n}{r_B^n} 
    \varphi_B\left( (S_1^* + \sqrt{2}S_2^*)^n
                  (S_1 + \sqrt{2}S_2)^n \right) \\ 
& = \frac{1}{2^{n}r_B^n} 
    \varphi_B\left( 
         \frac{1}{3}\{  2^{n+1} + (-1)^n\} 
       + \frac{1}{3}\{  2^{n+1} + 2(-1)^{n-1}\}S_1S_1^*   \right) \\                  
& = \frac{1}{3\cdot r_B^n} 
         \left( 
           2  + \frac{(-1)^n}{2^n} 
          \{2  - \frac{(-1)^n}{2^{n-1}} \} \frac{1}{r_B}      \right).                  
\end{align*}
We thus have
\begin{equation*}
- \lim_{n\to\infty}\frac{1}{n}\log \varphi_B(r_A^{-c_2^n})
= - \lim_{n\to\infty}\frac{1}{n}\log\frac{1}{r_B^n}
= \log r_B 
=h_{top}(\sigma_B).
\end{equation*}

\medskip

{\it Acknowledgments:}
The author would like to deeply thank
the referee for careful reading and lots of helpful advices 
in the presentation of the paper. 
This work was supported by 
JSPS KAKENHI Grant Numbers 15K04896, 19K03537.


\end{document}